\documentclass[12pt,a4paper]{amsart}
\usepackage{url}
\usepackage{amsmath,amsthm,amsfonts,amssymb,latexsym}

\newtheorem{theorem}{Theorem}[section]
\newtheorem{proposition}[theorem]{Proposition}
\newtheorem{lemma}[theorem]{Lemma}
\newtheorem{claim}[theorem]{Claim}

\newtheorem{corollary}[theorem]{Corollary}
\newtheorem{observation}[theorem]{Observation}
\theoremstyle{definition}

\newtheorem{question}[theorem]{Question}

\newcommand{\U}{\mathcal U}
\newcommand{\w}{\omega}

\newcommand{\IQ}{\mathbb Q}

\newcommand{\IP}{\mathbb P}

\newcommand{\K}{\mathcal{K}}

\newcommand{\G}{\mathcal{G}}
\newcommand{\I}{\mathcal{I}}

\newcommand{\F}{\mathcal{F}}

\newcommand{\V}{\mathcal{V}}

\newcommand{\uhr}{\upharpoonright}

\newcommand{\name}[1]{\dot{#1}}
\newcommand{\la}{\langle}
\newcommand{\ra}{\rangle}

\newcommand{\forces}{\Vdash}
\newcommand{\supp}{\mathrm{supp}}

\newcommand{\nothing}[1]{}

\title{Menger remainders of topological groups}

\author{Angelo Bella, Se{\c{c}}$\dot{\mathrm{\i}}$l Tokg\"oz,  and Lyubomyr
Zdomskyy}

\address{Department of Mathematics and Computer Science,
University of Catania,
 Citt\'a universitaria, viale A. Doria 6, 95125
Catania, Italy.}
 \email{bella@dmi.unict.it}

\address{Hacettepe University,
Faculty of Science, Department of Mathematics, 06800
Beytepe--Ankara, Turkey.} \email{secil@hacettepe.edu.tr}

\address{Kurt G\"odel Research Center for Mathematical Logic,
University of Vienna, W\"ahringer Stra\ss e 25, A-1090 Wien,
Austria.} \email{lzdomsky@gmail.com}
\urladdr{http://www.logic.univie.ac.at/\~{}lzdomsky/}

\thanks{ The first author's research that led to the present paper was partially
supported by a grant of the group GNSAGA of INdAM.
The second author would like to thank Hacettepe University
 BAP project 014 G 602 002 for its support.
The third author would like to thank the Austrian Academy
of Sciences (APART Program) as well as the Austrian Science Fund FWF
(Grant I 1209-N25) for generous support for this research. }


\begin{document}
\begin{abstract}
In this paper we discuss what kind of constrains combinatorial
covering properties of Menger, Scheepers, and Hurewicz impose on
remainders of topological groups. For instance, we show that such
a
remainder is Hurewicz if and only it is $\sigma$-compact.
Also,
the
existence of a Scheepers non-$\sigma$-compact remainder of a
topological group  follows from CH and yields a $P$-point, and
hence
is independent of ZFC. We also make an attempt to prove a
dichotomy
for the Menger property of remainders of topological groups in
the
style of Arhangel'skii.
\end{abstract}
\maketitle

\section{Introduction}
All topological spaces  are assumed to be completely regular. All
undefined topological notions can be found in \cite{Eng}. For a
space $X$ and its compactification $bX$ the complement
$bX\setminus
X$ is called a \emph{remainder} of $X$.
 The
interplay between the properties of spaces and their remainders
has
been studied since more than 50 years and resulted in a number of
duality results describing properties of $X$ in terms of those of
their remainders. A typical example of such a duality is the
celebrated result of Henriksen and Isbell stating that a
topological
space $X$ is Lindel\"of if and only if all (equivalently any) of
its
remainders is of \emph{countable type}, that is, any compact
subspace can be enlarged to another compact subspace with
countable
outer base.

In the last years, remainders in compactifications of topological
groups have been a popular topic. This is basically due to the
fact
that topological groups are much more sensitive to the properties
of
their remainders than topological spaces in general.  A major
role
in this study  was played  by Arhangel$'$skii, who initiated a
systematic study of this topic.  Among many other things, he
obtained two  elegant results, which are dichotomies for non-locally compact
topological groups.

\begin{theorem}[\cite{Arh09}] \label{ar1}  Let $G$ be a
topological group. If  
$bG$ is a compactification of $G$,  then $bG\setminus G$ is
either Lindel\"of or pseudocompact.
\end{theorem}
\begin{theorem}[\cite{Arh08}] \label{ar2} Let $G$ be a
topological group. If  $bG$ is a
compactification of $G$,  then $bG\setminus G$ is 
either
$\sigma$-compact or Baire.
\end{theorem}

We  recall that a  topological space $X$ is \emph{Baire}  if the
intersection of countably many open dense subsets is dense.

In this paper we will focus our attention  on  topological
properties which are strictly in between $\sigma$-compact and
Lindel\"of. Recall from \cite{COC1} that a space $X$ is
\emph{Menger} (or has the \emph{Menger property}) if  for any
sequence $(\U_n)_{n\in\w}$ of open covers of $X$ one may pick
finite
sets $\mathcal V_n\subset \mathcal U_n$ in such a way that
$\{\bigcup\mathcal V_n:n\in\omega\}$ is a cover of $X$. A family
$\{W_n:n\in\w\}$ of subsets of $X$ is called an $\w$-cover (resp.
$\gamma$-cover) of $X$, if for every $F\in [X]^{<\w}$ the  set
$\{n\in\w:F\subset W_n\}$ is infinite (resp. co-finite). The
properties of \emph{Scheepers} and \emph{Hurewicz} are defined in
the same way as the Menger property, the only difference being
that
we additionally demand that $\{\bigcup\mathcal V_n:n\in\omega\}$
is
a $\w$-cover (resp. $\gamma$-cover) of $X$. It is immediate that
$$\sigma\mbox{-compact} \Rightarrow \mbox{Hurewicz}\Rightarrow
\mbox{Scheepers} \Rightarrow \mbox{Menger} \Rightarrow
\mbox{Lindel\"of}.$$ The properties mentioned above have recently
received great attention, mainly because of their combinatorial
nature and  game-theoretic characterizations. One of the most
striking results about the Menger property is due to Aurichi who proved
\cite{Aur10} that  any Menger space is a $D$-space.

Our initial idea was to find counterparts of the properties
of
Menger, Scheepers, and Hurewicz in the style of
Theorems~\ref{ar1}
and \ref{ar2}. It turned out that the counterpart of the Hurewicz
property is already given by Theorem~\ref{ar1} because of the
following result, see Section~\ref{hurewicz} for its proof.

\begin{theorem} \label{main_hur}
Let $G$ be a topological group. If $\beta G\setminus G$ is
Hurewicz,
then it is $\sigma$-compact.
\end{theorem}
Let us note that there are ZFC examples of Hurewicz sets of reals
which are not $\sigma$-compact (see \cite[Theorem~5.1]{COC2} or
\cite[Theorem~2.12]{Tsa11}),
and thus Theorem~\ref{main_hur} is specific for remainders of
topological groups.

As it follows from the theorems below, which are the main results
of
this paper, for the properties of Scheepers and Menger
 the situation depends on the ambient set-theoretic
universe.
Each subspace of $\mathcal P(\w)$ (e.g., an ultrafilter)
is considered with the subspace topology. Let us recall from
\cite[Theorem~3.9]{COC2} that
if all finite powers of a topological space $X$ are Menger then
$X$ is Scheepers.
The converse of this statement fails consistently: under CH there
exists a Hurewicz subspace of
$\mathcal P(\w)$ whose square is not Menger, see
\cite[Theorem~3.7]{COC2}.

\begin{theorem} \label{main_sch}
There exists a Scheepers ultrafilter  iff there exists a
topological
group  $G$ such that $\beta G\setminus G$ is Scheepers and not
$\sigma$-compact iff there exists a
topological
group  $G$ such that all finite powers of $\beta G\setminus G$
are Menger   and not
$\sigma$-compact.
\end{theorem}

\begin{corollary} \label{cor_sch}
The existence of a topological group $G$ such that $\beta
G\setminus
G$ is Scheepers (resp. has all finite powers Menger)  
and not $\sigma$-compact is independent from ZFC.
More precisely, such a group exists under  $\mathfrak d=\mathfrak
c$, and its existence yields $P$-points.
\end{corollary}

Theorem~\ref{main_sch} and Corollary~\ref{cor_sch} are proved in
Section~\ref{scheepers}. Let us note that there exists a ZFC
example of 
a dense Baire subspace $X$ of $[\w]^\w$ all of whose finite
powers are Menger
(and thus also Scheepers), and hence it is indeed essential
in Theorem~\ref{main_sch} and Corollary~\ref{cor_sch} that we
consider remainders of
topological groups. In fact, such a subspace $X$ can be chosen to
be a filter,
see  \cite[Claim~5.5]{ChoRepZdo??} 
 and the proof of  \cite[Theorem~1]{RepZdoZha14}.

It is worth mentioning here that Scheepers ultrafilters have been studied
intensively  under different names for decades: In \cite{Can88} Canjar proved that 
$\mathfrak d=\mathfrak c$ implies the existence of an ultrafilter whose 
Mathias forcing does not add dominating reals. By \cite[Theorem~1.1]{ChoRepZdo??} 
the latter property for filters on $\w$ is equivalent to the Menger one, and 
by \cite[Claim~5.5]{ChoRepZdo??} all finite powers of a Menger filter are Menger.
Combining this with \cite[Theorem~3.9]{COC2} we conclude that for filters on $\w$,
the Menger property is equivalent to the Scheepers one, and hence Scheepers
(equivalently Menger [in all finite powers]) ultrafilters are exactly those studied
in \cite{Can88}. Another descriptions of such ultrafilters may be found 
in \cite{BlaHruVer13} and \cite{HruMin??}, where they were characterized  as ultrafilters
with certain topological and combinatorial properties stronger than being 
a $P$-point, respectively.  In particular, there are no Scheepers ultrafilters 
in models of ZFC without $P$-points.

Regarding the Menger property, we have the following partial
result established in Section~\ref{menger}. 
Note that the assumption on the remainder we make in it is
formally weaker than that
made in Theorem~\ref{main_sch}, see the last equivalent statement
there. 

\begin{theorem} \label{main_men}
It is consistent that for any topological group $G$ and
compactification $bG$, if $(bG\setminus G)^2$ is Menger, then it
is
$\sigma$-compact.
\end{theorem}

Let us note that  the properties of Menger, Scheepers, Hurewicz, having
Menger square, etc., are preserved by perfect maps in both
directions.  This implies that if one of
the remainders of a space $X$ has one of these covering properties,
then all others also have it, see the beginning of Section~\ref{scheepers}
for more details and corresponding definitions. Thus Theorems~\ref{main_hur},
\ref{main_sch}, and \ref{main_men} admit several equivalent reformulations,
given by the freedom to consider either all or some (specific) compactifications.


In light of Theorems~\ref{main_men} and \ref{main_sch} it is
natural
to ask the following questions.

\begin{question}
Is there a ZFC example of a topological group with a Menger
non-$\sigma$-compact remainder?
\end{question}

\begin{question}
Is it consistent that there exists a topological group $G$ such
that
$\beta G\setminus G$ is Menger and not Scheepers? Does CH imply
the
existence of such a group?
\end{question}

Since we do not have an analogous statement to
Theorem~\ref{main_sch} for the Menger property (in
Theorem~\ref{main_men} we make a somewhat unpleasant assumption
that
the square of the remainder is Menger), it may still be the case
that for the Menger property there exists a dichotomy similar to
Theorems~\ref{ar1} and \ref{ar2}. In
Section~\ref{dichotomy_menger}
we analyze some properties which might be counterparts of the
Menger
one for remainders of topological groups.

\section{Hurewicz remainders} \label{hurewicz}

According to the definition  on \cite[p. 235]{ArkTka11}, a
topological group $G$ is \emph{feathered} if it contains a
non-empty
compact subspace with countable outer base. Recall that a family $\mathcal U$
of open subsets of a topological space $X$ is an \emph{outer base} for a subset
$A$ of $X$ if $A\subset U$ for all $U\in\mathcal U$, and for every
open $O\supset A$ there exists $U\in\mathcal U$ such that $U\subset O$.
 By
\cite[Lemma~4.3.10]{ArkTka11} every feathered group has a compact
subgroup with countable outer base. A topological group $G$
is \emph{Raikov-complete} if it is complete in the uniformity
generated by sets $\{(x,y)\in G^2:xy^{-1}, x^{-1}y\in U\}$,
where $U$ is a neighbourhood of the neutral element of $G$, see
\cite[\S~3.6]{ArkTka11} and references therein.

The following fact is
probably
well-known.

\begin{lemma} \label{feathered}
For a  feathered group $G$ the following conditions are
equivalent:
\begin{itemize}
\item[$(1)$] $G$ is \v{C}ech-complete;
\item[$(2)$] Each closed subgroup $G_0$ of $G$
admitting a dense $\sigma$-compact subspace is \v{C}ech-complete;
\item[$(3)$] There exists a compact subgroup $H$ of $G$
with countable outer base such that $\overline{\la QH\ra}$ is
\v{C}ech-complete for every countable $Q\subset G$,  where for
$X\subset G$ we denote by $\la X\ra$ the smallest subgroup of $G$
containing $X$.
\end{itemize}
\end{lemma}
\begin{proof}
The implication $(1)\to(2)$ is straightforward,
 and  $(2)\to (3)$ is a direct consequence of
\cite[Prop.~4.3.11]{ArkTka11}.
The proof of $(3)\to (1)$ will be  obtained by   a tiny
modification
of that of \cite[Theorem~4.3.15]{ArkTka11}. In particular, all details missing 
below can be found in the proof of the above-mentioned theorem.

Let $\la V_n:n\in\w\ra$ be a decreasing sequence of open symmetric 
neighbourhoods of $H$ such that $V_{n+1}^2\subset V_n$ for all $n$
and $H=\bigcap_{n\in\w}V_n$. By \cite[Lemma~3.3.10]{ArkTka11}
there exists a continuous prenorm\footnote{Following \cite[\S~3.3]{ArkTka11}
we call a function $N:G\to\mathbb R_+$ a \emph{prenorm}, if 
$N(e)=0$, $N(x^{-1})=N(x)$, and $N(xy)\leq N(x)+N(y)$ for all $x,y\in G$.} on $G$
which satisfies 
$$ \{x\in G: N(x)<1/2^n\}\subset V_n\subset \{x\in G: N(x)<2/2^n\} $$
for all $n\in\w$. Thus $H=\{x\in G:N(x)=0\}$. Let $\rho$ be a pseudometric
on $G$ defined by $\rho(x,y)=N(xy^{-1})+N(x^{-1}y)$ for all $x,y\in G$.
Then $\rho$ is continuous and $\rho(x,y)=0$ iff $xy^{-1},x^{-1}y\in H$.
Consider the equivalence relation $\sim$ on $G$ defined by $x\sim y$ iff $\rho(x,y)=0$
and denote by $X$ the quotient space $G/\sim$. Let $\pi:G\to X$
be the quotient map and set $\rho^*(\pi(x),\pi(y))=\rho(x,y)$
for any $x,y\in G$. It follows that $\pi$ is perfect , $\rho^*$ is well-defined, and is a 
metric generating the quotient topology on $X$. 

The proof of \cite[Theorem~4.3.15]{ArkTka11} is done as follows:
Assuming that $G$ is Raikov-complete, it is shown that $\rho^*$
is complete. The same argument,
applied to any subgroup $G'$ of $G$ which contains $H$ (and hence is closed under 
$\sim$) yields that if $G'$ is Raikov-complete, then
$\rho^*\uhr \pi[G']$ is complete. Our item $(3)$ implies that
$\overline{\la QH\ra}$ is Raikov-complete for every countable $Q\subset G$
because \v{C}ech-complete groups are Raikov-complete, see, e.g., 
\cite[Theorem~4.3.7]{ArkTka11}. Therefore by $(3)$ we have that 
$ \rho^*\uhr \pi[\overline{\la QH\ra}]  $
is complete for any countable $Q\subset G$. Since $\pi$ is perfect,
the latter gives that $\rho^*\uhr Y$ is complete for any separable closed
$Y\subset X$, and hence $\rho^*$ is complete. 
  therefore $G$ is \v{C}ech-complete
being
a perfect preimage of a complete metric space, see \cite[Theorems~3.9.10 and 4.3.26]{Eng}.
\end{proof}

We are in a position now to present the
\medskip

\noindent\textit{Proof of Theorem~\ref{main_hur}.} \ Let $G$ be a
topological group such that $\beta G\setminus G$ is Hurewicz.
Then
$\beta G\setminus G$ is Lindel\"of, hence $G$ is of countable
type \cite{HenIsb57},
and therefore it is feathered. By Lemma~\ref{feathered} it is
enough
to show that
 each closed subgroup $G_0$ of $G$
admitting a dense $\sigma$-compact subspace is \v{C}ech-complete.
Let $G_0$ be as above, $F$ be a dense $\sigma$-compact subspace
of
$G_0$, and $X=\overline{G_0}\setminus G_0$, where the closure is
taken in $\beta G$. Since $X$ is closed in $\beta G\setminus G$,
it
is Hurewicz. Applying \cite[Theorem~27]{BanZdo06}\footnote{We
refer
here to the online version of the paper available from the
web-pages
of the authors, which is an extended version of the published
one.}
to the Hurewicz space $X$ and \v{C}ech-complete space
$\overline{G_0}\setminus F$ containing it, we conclude that there
exists a $\sigma$-compact space $F'$ such that $X\subset
F'\subset
\overline{G_0}\setminus F$, which implies that
$\overline{G_0}\setminus F'$ is a dense (because it contains $F$)
\v{C}ech-complete subspace of $G_0$. Thus $G_0$ is \v{C}ech-complete 
by \cite[Theorem~1.2]{Arh09}, which completes our proof. \hfill
$\Box$
\medskip

It is well-known \cite[Lemma~22]{Zdo06} that if player
$\mathit{II}$
has a winning strategy in the Menger game on a space $X$ (see
 Section~\ref{menger} for its definition) then $X$ is Hurewicz.
 Therefore Theorem~\ref{main_hur} generalizes \cite[Corollary
 3.5]{AurBel??}.

\section{Scheepers remainders}\label{scheepers}

In the proof of Theorem~\ref{main_sch}, which is the main goal of
this section, we shall need set-valued maps, see \cite{RepSem98}
for
more information on them. By a \emph{set-valued map} $\Phi$ from
a
set $X$ into a set $Y$ we understand a map from $X$ into
$\mathcal
P(Y )$ and write $\Phi : X\Rightarrow Y$  (here $\mathcal P(Y)$
denotes the set of all subsets of $Y$). For a subset $A$ of $X$
we
set $\Phi(A) = \bigcup_{x\in A}\Phi(x) \subset Y$. A set-valued
map
$\Phi$ from a topological space $X$ to a topological space $Y$
is
said to be
\begin{itemize}
\item  \emph{compact-valued}, if $\Phi(x)$  is compact for every
$x \in X$;
\item \emph{upper semicontinuous}, if for every open subset $V$
of $Y$ the
set $\Phi^{-1} (V) = \{x \in  X : \Phi(x) \subset V\}$ is open in
$X$.
\end{itemize}
To abuse terminology, we shall call compact-valued upper
semicontinuous maps \emph{cvusc} maps. It is known
\cite[Lemma~1]{Zdo05} that all combinatorial covering properties
considered in this paper are preserved by cvusc maps. Also, if
$f:X\to Y$ is a perfect map, then $f^{-1}:Y\Rightarrow X$
assigning
to $y\in Y$ the subset $f^{-1}(y)$ of $X$, is a cvusc maps.
Therefore the properties of Menger, Scheepers, Hurewicz, having
Menger square, etc., are preserved by perfect maps in both
directions. That is, if $f$ is perfect and $Z\subset X$ (resp.
$Z\subset Y$) has one of these properties, then so does $f(Z)$
(resp. $f^{-1}(Z)$). In particular, this implies that if one of
the
remainders of a space $X$ has one of these covering properties,
then
all others also have it.  In addition, all these properties are
preserved
by
product with $\w$ equipped with the discrete
topology\footnote{For
the Scheepers property this fact  is slightly non-trivial and
follows
from \cite[Proposition~4.7]{TsaZdo08}.}.

We shall also need some additional notation. For
$X\subset\mathcal
P(\w)$ we shall denote by $\sim X$ the set $\{\w\setminus x:x\in
X\}$. Note that $\sim X$ is homeomorphic to $X$ because
$x\mapsto\w\setminus x$ is a homeomorphism from $\mathcal P(\w)$
to
itself. For subsets $a,b$ of $\w$ (resp. $a,b\in\w^\w$)
$a\subset^*
b$  (resp. $a\leq^* b$) means $|a\setminus b|<\w$ (resp.
$|\{n:a(n)>b(n)\}|<\w$). A collection $\F$ of infinite subsets of
$\w$ is called a \emph{semifilter} if for any $a\in\F$ and
$a\subset^* b$ we have $b\in\F$. For a semifilter $\F$ we set
$\F^+=\{x\subset\w:\forall a\in\F (a\cap x\neq\emptyset)\}$. Note
that $\F^+=\mathcal P(\w)\setminus\sim\F$. $\mathfrak Fr$ denotes
the minimal with respect to inclusion semifilter which consists
of
all co-finite sets. For the other notions used in the proof of
the
following statement we refer the reader to \cite{ArkTka11}.

\begin{lemma} \label{towards_semifilter}
Suppose that $G$ is a  topological group, $K$ is a compact
subgroup
of $G$ with countable outer base in $G$, and $QK$ is dense in $G$
for some countable $Q\subset G$. Let $\mathsf P$ be a property of
topological spaces preserved by images under cvusc maps and
product
with $\w$ equipped with the discrete topology.

 If $\beta G\setminus G$ has $\mathsf P$
 and is not $\sigma$-compact, then there exists
a semifilter  $\F$ such that $\F^+\subset\F$ and $\F$ has
$\mathsf
P$. If, moreover, $(\beta G\setminus G)^2$ is Menger, then there
exists a semifilter $\F$ such that $\F=\F^+$ and $\F^2$ is 
Menger.
\end{lemma}
\begin{proof}
Observe that $G$  is not locally compact because otherwise its
remainders would be compact.
 Since $G$ is feathered,  there exists a
\v{C}ech-complete group $\tilde{G}$
 containing $G$ as a dense subgroup, see
\cite[Theorem~4.3.16]{ArkTka11}. Let
$\beta\tilde{G}$ be the Stone-\v{C}ech
  compactification
of $\tilde{G}$. It follows that $\beta\tilde{G}\setminus G$ is
not
$\sigma$-compact, and hence $G\neq\tilde{G}$. Fix
$g\in\tilde{G}\setminus G$ and note that $gG$ is dense in
$\beta\tilde{G}$ and $gG\cap G=\emptyset$. Therefore both
$\beta\tilde{G}\setminus G$ and $\beta\tilde{G}\setminus gG$ have
property $\mathsf P$ being remainders of spaces homeomorphic to
$G$,
and $\beta\tilde{G}=(\beta\tilde{G}\setminus G)\bigcup
(\beta\tilde{G}\setminus gG)$.

Note that $K$ has a countable outer base also in $\tilde{G}$, and
hence the quotient space $X:=\tilde{G}/K=\{zK:z\in\tilde{G}\}$ is
metrizable. It is also separable by our assumption on $G$, so
there
exists a metrizable compactification $bX$ of $X$. In addition,
the
quotient map $\pi_K:\tilde{G}\to X$, $\pi_K(z)=zK$, is perfect by
\cite[Theorem~1.5.7]{ArkTka11}, and hence by
\cite[Theorem~3.7.16]{Eng} it can be extended to a (perfect) map
$\pi:\beta\tilde{G}\to bX$ such that
$\pi(\beta\tilde{G}\setminus\tilde{G})=bX\setminus X$. 
$\pi\uhr\tilde{G}=\pi_K$, hence $G=\pi^{-1}(\pi(G))$,
$gG=\pi^{-1}(\pi(gG))$, and consequently
$A:=\pi(\beta\tilde{G}\setminus gG)$ and
$B:=\pi(\beta\tilde{G}\setminus G)$ are both co-dense subsets of
$bX$
with property $\mathsf P$ covering $bX$.

 Note that $bX$ has no isolated points because both
$G=\pi^{-1}(\pi(G))$
 and $\beta\tilde{G}\setminus
G=\pi^{-1}(\pi(\beta\tilde{G}\setminus G))$
 are nowhere locally compact. Since $bX$ is a metrizable compact,
 there exists a continuous surjective map $f:\mathcal P(\w)\to
X$. Applying
 \cite[3.1.C(a)]{Eng} we can find a closed subspace $T$ of
$\mathcal P(\w)$
 such that $f\uhr T\to bX$ is surjective and irreducible, i.e.,
 $f[T']\neq bX$ for any closed $T'\subsetneq T$.  $T$ has no
 isolated points: if $t\in T$ were isolated then the
irreducibility
 of $f\uhr T$ would give that $t=(f\uhr T)^{-1}[f(t)]$,
 which infers that $f(t)$ is isolated in $bX$ and this leads to a
contradiction.
Therefore $T$ is homeomorphic to $\mathcal P(\w)$, and hence
there
exists a continuous surjective irreducible $h:\mathcal P(\w)\to
bX$.
Since $h$ is  irreducible, both $C:=h^{-1}(A)$ and $D:=h^{-1}(B)$
are
co-dense. Since $h$ is perfect, they also have $\mathsf P$: both
of
them are cvusc images of $\beta\tilde{G}\setminus G$. Note  also
that $\mathcal P(\w)=C\bigcup D$.

The Cantor set $\mathcal P(\w)$ has the following fundamental
property (see \cite{ArhMil14} and references therein) which can
be
directly proved by Cantor's celebrated back-and-forth argument:
for
any countable dense subsets $I_0,I_1,J_0,J_1$ of $\mathcal P(\w)$
such that $I_0\cap I_1=\emptyset$ and $J_0\cap J_1=\emptyset$
there
exists a homeomorphism $i:\mathcal P(\w)\to\mathcal P(\w)$ with
the
property $i(I_0)=J_0$ and $i(I_1)=J_1$. Therefore there is no
loss
of generality to assume that $[\w]^{<\w}\subset \mathcal
P(\w)\setminus C$ and $\mathfrak Fr\subset \mathcal
P(\w)\setminus
D$. Set
$$\F_0=\{x\subset\w:\exists c\in C,u\in [\w]^{<\w}, v\subset\w \:
(x=(c\setminus u)\cup v) \},$$
$$\I_0=\{x\subset\w:\exists d\in D, u\in [\w]^{<\w}, v\subset\w
\: (x=(d\cup u)\setminus v) \},$$
and note that both $\F_0$ and $\sim\I_0$ are semifilters. It
follows
that both $\F_0$ and $\I_0$ are countable unions of continuous
images of $C\times\mathcal P(\w)$ and $D\times\mathcal P(\w)$,
respectively, and consequently they are cvusc images of
$(\beta\tilde{G}\setminus G)\times\w$, where $\w$ is considered
with
the discrete topology. Since the property $\mathsf P$  is
preserved
by product with $\w$, we conclude that both $\F_0$ and $\I_0$
have
it.

 Set $\F=\F_0\cup\sim\I_0$ and note
that it has property $\mathsf P$ for the same reason as
$\F_0,\I_0$
do. Since $C\subset\F_0\subset\F$ and
$D\subset\I_0=\sim(\sim\I_0)\subset\sim\F$, we have that
$\mathcal
P(\w)=\F\bigcup\sim\F$. Therefore $\F^+\subset\F$ because
$\F^+=\mathcal P(\w)\setminus\sim\F$.
\smallskip

To prove the ``moreover'' part assume that $\F^2$ is Menger and
consider the following map $\phi:(\F\cap \sim\F)\to [\w]^{\w}$:
$$ \phi(a)=(a\cup\{n+1:n\in a\})\setminus a.$$
Note that $\F\cap\sim\F$ is homeomorphic to the closed subset
$\{(x,x):x\in\mathcal P(\w)\}\cap (\F\times\sim \F)$ of the
Menger
space $\F\times\sim\F$ and thus is Menger itself. Therefore
$\phi(\F\cap\sim\F)$ is not dominating.

Every strictly increasing sequence $\bar{k}=(k_n)_{n\in\w}$ of
integers such that $k_0=0$ generates a monotone surjection
$\psi_{\bar{k}}:\w\to\w$ by letting
$\psi_{\bar{k}}^{-1}(n)=[k_n,k_{n+1})$. We claim that there
exists
$\bar{k}$ as above such that $\psi_{\bar{k}}(\F)^+=
\psi_{\bar{k}}(\F)$. Suppose to the contrary that for every
$\bar{k}$ there exists $a_{\bar{k}}\in\F$ such that
$\psi_{\bar{k}}(a_{\bar{k}})\in \psi_{\bar{k}}(\F)\setminus
\psi_{\bar{k}}(\F)^+$, i.e.,
$\w\setminus\psi_{\bar{k}}(a_{\bar{k}})\in\psi_{\bar{k}}(\F)$.
Then
both
$b_{\bar{k}}:=\psi_{\bar{k}}^{-1}(\psi_{\bar{k}}(a_{\bar{k}}))$
and $\w\setminus b_{\bar{k}}$ are in $\F$, and therefore
$b_{\bar{k}}\in\F\cap\sim\F$. Note, however, that
$\phi(b_{\bar{k}})\subset\{k_n:n\in\w\}$, which means that
$\bar{k}\leq^*\phi(b_{\bar{k}})$. Since $\bar{k}$ was chosen
arbitrarily we get that $\phi(\F\cap\sim\F)$ is dominating, which
is
impossible. This contradiction implies that $\psi(\F)^+=\psi(\F)$
for some monotone surjection $\psi:\w\to\w$, and then $\psi(\F)$
is
the  semifilter with Menger square we were looking for.
\end{proof}

Recall that a family $\mathcal X\subset\mathcal P(\w)$ is
\emph{centered} if $\bigcap\mathcal X'$ is infinite for every
$\mathcal X'\in[\mathcal X]^{<\w}$. We are in a position now to
present the
\medskip

\noindent\textit{Proof of Theorem~\ref{main_sch}.} If $\F$ is a
Scheepers ultrafilter, then $\sim\F$ is a subgroup of $(\mathcal
P(\w),\Delta)$ and $\F\cup\sim\F=\mathcal P(\w)$. Thus $\F$ is a
Scheepers non-$\sigma$-compact (no ultrafilter can be Borel)
remainder of the group $(\sim\F,\Delta)$. Moreover, all finite
powers of
$\F$ are Menger (and hence also Scheepers) by
\cite[Claim~5.5]{ChoRepZdo??}. 

Let us now prove the ``if'' part, i.e., assume that $(\beta
G\setminus G)$
is Scheepers and not $\sigma$-compact. In the same way as at the
beginning of the proof of Theorem~\ref{main_hur} we conclude that
$G$ is feathered. By Lemma~\ref{feathered} we may assume without
loss of generality that $G$ satisfies the premises of
Lemma~\ref{towards_semifilter}. Applying this lemma for $\mathsf
P$
being the Scheepers property, we conclude that
 there exists a Scheepers
semifilter  $\F$  such that $\F^+\subset\F$. For every $n\in\w$
let
us denote by $O_n$ the open subset $\{x\subset\w:n\in x\}$ of
$\mathcal P(\w)$ and note that each $x\in \F$ belongs to
infinitely
many members of
  $\U_0=\{O_n:n\in\w\}$. Applying \cite[Theorem~21]{SamSchTsa09}
(namely the implication $(1)\to(2)$ there) we conclude that there
exists an increasing number sequence $(n_k)_{k\in\w}$ such that
$n_0=0$ and
$$\big\{\{\bigcup O_n:n\in[n_k,n_{k+1}\}:k\in\w\big\}$$
is an $\w$-cover of $\F$. The latter means that for any family
$\{x_0,\ldots,x_l\}\subset \F$ there exist infinitely many
$k\in\w$
such that $x_i\cap [n_k,n_{k+1})\neq\emptyset$ for all $i\leq l$.

Let us define $\phi:\w\to\w$ by letting
$\phi^{-1}(k)=[n_k,n_{k+1})$
for all $k$ and set
$$\mathcal
S=\{s\subset\w:\phi^{-1}(s)\in\F\}=\{\phi(x):x\in\F\}.$$
Then $\mathcal S$ is a Scheepers semifilter being a continuous
image
of $\F$.
\begin{claim}\label{sch_01}
$\mathcal S$ is centered.
\end{claim}
\begin{proof}
 Given any $s_0,\ldots, s_l\in\mathcal S$, set
$x_i=\phi^{-1}(s_i)$
and take $k\in\w$ such that $x_i\cap[n_k,n_{k+1})=
x_i\cap\phi^{-1}(k)\neq\emptyset$ for all $i\leq l$. There are
infinitely many such $k$'s, and each of them is an element of
$s_i$
for all $i$ because $s_i=\phi(x_i)$.
\end{proof}
\begin{claim}\label{sch_02}
 $\mathcal S^+\subset\mathcal S$.
\end{claim}
\begin{proof}
 Take any $x\in\mathcal S^+$ and set $y=\phi^{-1}(x)$. Then
$y\in\F^+$:
 given $u\in\F$, note that $\phi(u)\in\mathcal S$, and hence
$|\phi(u)\cap x|=\w,$ which implies that $|u\cap y|=\w$ and thus
$y$
meets all elements of $\F$. Since $\mathcal F^+\subset\F$, we
have
that $y\in\F$, and consequently $x=\phi(y)\in\mathcal S$.
\end{proof}
Let us fix now $s_0,\ldots, s_i\in\mathcal S$ and take arbitrary
$s\in\mathcal S$. Claim~\ref{sch_01} implies that
$|s\cap\bigcap_{i\leq l}s_i|=\w$,  hence $\bigcap_{i\leq
l}s_i\in\mathcal S^+$, and therefore $\bigcap_{i\leq
l}s_i\in\mathcal S$ by Claim~\ref{sch_02}. Thus $\mathcal S$ is a
filter, and consequently it is an ultrafilter by
Claim~\ref{sch_02}.
This completes our proof. \hfill $\Box$
\medskip

We call a semifilter $\F$ a \emph{$P$-semifilter} if for every
sequence $(F_n)_{n\in\w}\in\F^\w$ there exists a sequence
$(A_n)_{n\in\w}$ such that $A_n\in [F_n]^{<\w}$ and
$\bigcup_{n\in\w}A_n\in\F$. Note that if $\F$ is a filter then we
get a standard definition of a $P$-filter. $P$-filters which are
ultrafilters are nothing else but $P$-points.

 Recall that for every
$n\in\w$ we  denote by $O_n$ the clopen subset $\{x\subset\w:n\in
x\}$ of $\mathcal{P}(\w)$.  The following fact is
straightforward.

\begin{observation} \label{obs_obv}
Let $A\subset\w$ and $\F$ be a semifilter. Then $\{O_n:n\in A\}$
covers $\F^+$ iff $A\in\F$. Consequently, if $\F^+$ is Menger,
then
$\F$ is a $P$-semifilter.
\end{observation}

\noindent\textit{Proof of Corollary~\ref{cor_sch}.} \ It is known
\cite{Can88} that  under $\mathfrak d=\mathfrak c$ there exists
an
 ultrafilter $\F$ on $\w$ such that the
Mathias forcing $\mathbb M(\F)$ does not  add dominating reals,
see
\cite{Can88} for corresponding definitions. Applying
\cite[Theorem~1]{ChoRepZdo??} we conclude that $\F$ is Menger
when
considered with the topology inherited from $\mathcal P(\w)$. By
\cite[Claim~5.5]{ChoRepZdo??} we have that all finite powers of
$\F$
are Menger, and hence $\F$ is Scheepers by
\cite[Theorem~3.9]{COC2}.

Now suppose that $\F$ is a Menger ultrafilter. By the maximality
of
$\F$ we have $\F=\F^+$. Now it suffices to apply
Observation~\ref{obs_obv}.
 \hfill $\Box$

\section{Menger remainders} \label{menger}
This section is devoted to the proof of Theorem~\ref{main_men}
which
is divided into a sequence of lemmata. In the proof of the next
lemma we shall need the following game of length $\omega$ on a
topological space $X$:
 In the $n$th  move player $I$ chooses an open cover
$\U_n$ of $X$, and player $\mathit{II}$ responds by choosing a
finite $\V_n\subset \U_n$. Player $\mathit{II}$ wins the game if
$\bigcup_{n\in\omega} \bigcup\V_n =X$. Otherwise, player $I$
wins.
We shall call this game
 \emph{the Menger game} on $X$.
 It is well-known  that
  $X$  is Menger  if and only if
 player $I$ has no winning strategy in the Menger game on
$X$, see \cite{Hur25} or \cite[Theorem~13]{COC1}.

Formally, a strategy for player $I$ is a map
$\S:\tau^{<\w}\to\mathcal O(X)$, where $\tau$ is the topology of
$X$
and $\mathcal O(X)$ is the family of all open covers of $X$. The
strategy $\S$ is winning if $\bigcup_{n\in\w}U_n\ne X $ for any
sequence $(U_n)_{n\in\w}\in\tau^\w$ such that $U_n$ is a union of
a
finite subset of $\S(U_0,\ldots, U_{n-1})$ for all $n\in\w$.


\begin{lemma} \label{bo_menger}
Suppose that $\F$ is a Menger semifilter. Then for every sequence
$\la B_i:i\in\w\ra \in (\F^+)^\w$ and increasing $h\in\w^\w$
there
exists increasing $\delta\in\w^\w$ such that
$$\bigcup_{i\in\w} B_i\cap [h(2\delta(i)),
h(2\delta(i+1)))\in\F^+.$$
 \end{lemma}
\begin{proof}
For every $n\in\w$ let us denote by $O_n$ the  subset
$\{x\subset\w:n\in x\}$ of $\mathcal P(\w)$ and note that $O_n$
is
clopen. It is easy to see that for $B\subset\w$ the collection
$\U_B:=\{O_n:n\in B\}$ is an open cover of $\F$ if and only if
$B\in\F^+$. Set $\delta(0)=0$ and
 consider the following strategy for player $\mathit{I}$ in the
Menger game on $\F$:
In the $0$th move he  chooses $\U_{B_0\setminus
h(0)}=\U_{B_0\setminus h(\delta(0))}$.
 Suppose that for some $i\in\w$ we have already defined
$\delta(i)$. Then
player $\mathit{I}$  chooses
 $\U_{B_i\setminus h(2\delta(i))}$. If player $\mathit{II}$
responds by choosing
$\V_i\in [\U_{B_i\setminus h(2\delta(i))}]^{<\w}$,  then we
define
$\delta(i+1)$ to be so that $\V_i\subset\{O_n:n\in
[h(2\delta(i)),
h(2\delta(i+1)))\cap B_i\}$, and the next move of player
$\mathit{I}$ is $\U_{B_{i+1}\setminus h(2\delta(i+1))}$.

The strategy for player $\mathit{I}$ we described above is not
winning, so there exists a run in the Menger game in which he
uses
this strategy and looses. Let $\delta$ be the function  defined
in
the course of this run. It follows that
$\bigcup\{\V_i:i\in\w\}\supset\F$, where the $\V_i$s are the
moves
of player $\mathit{II}$,  and hence
$$\bigcup_{i\in\w} B_i\cap [h(2\delta(i)),
h(2\delta(i+1)))\in\F^+$$
because $\V_i\subset\{O_n:n\in [h(2\delta(i)),
h(2\delta(i+1)))\cap
B_i\}$.
\end{proof}

For a semifilter $\F$ we denote by $\IP_\F$ the poset consisting
of
all partial maps $p$ from $\w\times\w$ to $2$ such that for every
$n\in\w$ the domain of $p_n:k\mapsto p(n,k)$ is an element of
$\sim\F$. If, moreover, we assume that and
$\mathrm{dom}(p_n)\subset
\mathrm{dom}(p_{n+1})$ for all $n$, the corresponding poset will
be
denoted by $\IP_\F^*$. A condition $q$ is stronger than $p$ (in
this
case we write $q\leq p$) if $p\subset q$. For filters $\F$ the
poset
$\IP_\F^*$ is obviously dense in $\IP_\F$, and the latter is
proper
and $\w^\w$-bounding if $\F$ is a non-meager $P$-filter
\cite[Fact
VI.4.3, Lemma~VI.4.4]{She_propimp}. In light of
Observation~\ref{obs_obv}, the following lemma may be thought of
as
a topological counterpart of \cite[Fact VI.4.3,
Lemma~VI.4.4]{She_propimp}.

\begin{lemma} \label{w^w-bound}
If $\F^+$ is a Menger semifilter, then both  $\IP_\F$ and
$\IP_\F^*$
are proper and $\w^\w$-bounding.
\end{lemma}
\begin{proof}
We shall present the proof for $\IP_\F$. The one for the poset
$\IP^*_\F$ is completely analogous.

To prove the properness let us fix a countable elementary
submodel
$M\ni\IP_\F$ of $H(\theta)$ for $\theta$ big enough, a condition
$p\in\IP_\F\cap M$, and list all open dense subsets of $\IP_\F$
which are elements of $M$ as $\{D_i:i\in\w\}$. Let us denote by
$\tau$ the collection of all open subsets of $\mathcal P(\w)$.
For
every $s\in [\w]^{<\w}$ we shall denote by $O_s$ the set
$\{x\subset\w:x\cap s\neq\emptyset\}$. $O_s$ is clearly a clopen
subset of $\mathcal P(\w)$.

In what follows we shall define a strategy $\S:\tau^{<\w}\to
\mathcal O(\F^+)$ of player $\mathit{I}$ in the Menger game on
$\F^+$ as well as a map $\S_0:\tau^{<\w}\cap M\to\IP_\F\cap M$.
Set
$p^0=p$, $\S_0(\emptyset)=p^0$,  and
$$\S(\emptyset)= \{ O_s: \exists l\in\w\: [s=(\w\setminus
\mathrm{dom}(p^0_0))\cap l]\}.$$
Now suppose that for some $n\in\w$ and all sequences
$(U_k)_{k<n}$
of open subsets of $\mathcal P(\w)$ we have defined
$p^n=\S_0((U_k)_{k<n})$ and $\S((U_k)_{k<n})$, and fix such a
sequence $(U_k)_{k\leq n}$ of length $n$.  If  $U_n$ is not of
the
form $\bigcap_{i\leq n} O_{s^n_i}$, where $s^n_i=
(\w\setminus\mathrm{dom}(p^n_i))\cap l$ for some $l\in\w$, then
$\S_0((U_k)_{k\leq n})$ and $\S((U_k)_{k\leq n})$ are irrelevant.
Otherwise write $\prod_{i\leq n} 2^{\{i\}\times s^n_i}$ in the
form
$\{(t^{n,j}_i)_{i\leq n}:j\leq N\}$, set $p^{n,-1}=p^n$, and by
induction on $j\leq N$ define a decreasing sequence
$(p^{n,j})_{j\leq N}$ of conditions in $\IP_\F\cap M$ with the
following properties:
\begin{itemize}
 \item[$(i)$] $\mathrm{dom}(p^{n,j}_i)\cap s^n_i=\emptyset$ for
all $j\leq N$
and $i\leq n$;
\item[$(ii)$] $p^{n,j} \cup\bigcup_{i\leq n}t^{n,j}_i\in D_n$ for
all $j\leq N$.
\end{itemize}
Then we let $p^{n+1}=p^{n,N}$,  $\S_0((U_k)_{k\leq n})=p^{n+1}$
and
$$\S((U_k)_{k\leq n})=\{ \bigcap_{i\leq n+1}O_{s^{n+1}_i}:
\exists l\in\w\forall i\leq n+1\:
\big[s^{n+1}_i=(\w\setminus \mathrm{dom}(p^{n+1}_i))\cap
l\big]\}.$$
Since $\F^+$ is Menger, $\S$ cannot be a winning strategy for
player
$\mathit{I}$, and hence there exists a sequence $(U_n)_{n\in\w}$
of
open subsets of $\mathcal P(\w)$ with the following properties:
\begin{itemize}
 \item[$(iii)$] $p^n:=\S_0((U_i)_{i<n})\in\IP_\F\cap M$ for all
$n\in\w$;
\item[$(iv)$] For every $n\in\w$ there exists $l_n\in\w$ such
that
$U_n=\bigcap_{i\leq n}O_{s^n_i}$, where $ s^n_i=(\w\setminus
\mathrm{dom}(p^n_i))\cap l_n$;
\item[$(v)$] $l_n\leq l_{n+1}$ for all $n\in\w$;
\item[$(vi)$] $\mathrm{dom}(p^{n+1}_i)\cap s^n_i=\emptyset$; and
\item[$(vii)$] $\F^+\subset\bigcup_{n\in\w} U_n$.
\end{itemize}
 Items $(iv)$ and $(vii)$ imply that for every $i\in\w$
we have $\F^+\subset\bigcup_{n\geq i}O_{s^n_i}$, therefore for
every
$x\in\F^+$ there exists $n\geq i$ such that $x\cap
s^n_i\neq\emptyset$, which is equivalent to $y_i:=\bigcup_{n\geq
i}s^n_i\in(\F^+)^+=\F$. By the definition of $s^n_i$ in $(iv)$
together with items $(v)$ and $(vi)$ we have that $s^n_i\subset
s^{n+1}_i$ for all $n\in\w$ and $i\leq n$, and hence
$y_i\cap\bigcup_{n\geq i}\mathrm{dom}(p^n_i)=\emptyset$. Since
$i\in\w$ was chosen arbitrarily, we conclude that
$q:=\bigcup_{n\in\w}p_n\in\IP_\F$.

We claim that $q$ is $(M,\IP_\F)$-generic. Indeed, pick $q'\leq
q$,
$n\in\w$, and $r\leq q'$ such that $\mathrm{dom}(r_i)\supset
s^n_i$
for all $i\leq n$. Then there exists $j\leq N$ such that
$r_i\uhr(\{i\}\times s^n_i)=t^{n,j}_i$ for all $i\leq n$, and
consequently
$$ r  \leq q \cup\bigcup_{i\leq n}t^{n,j}_i
\leq p^{n,j} \cup\bigcup_{i\leq n}t^{n,j}_i\in D_n $$ by $(ii)$.
This implies that $r$ is compatible with an element of $D_n\cap
M$
and thus completes our proof of the properness.

Note that for every $n$ we have found a finite subset $A_n$
(namely
$\{p^{n,j}:j\leq N\}$) of $D_n\cap M$ such that any extension of
$q$
is compatible with some element of $A_n$. If $\name{f}\in M$ is a
$\IP_\F$-name for a real, then    the open dense subset of
$\IP_\F$
consisting of those conditions which determine $\name{f}(k)$
equals
$D_{n_k}$ for some $n_k\in\w$. It follows from
 the above that $q$ forces that $\name{f}(k)$ cannot exceed
 $\max\{l:\exists u\in A_{n_k}\: (u\Vdash
\name{f}(k)=\check{l})\}$,
 and therefore $\IP_\F$ is $\w^\w$-bounding.
\end{proof}

For a relation $R$ on $\w$ and $x,y\in\w^\w$ we denote by
$[x\:R\:y]$ the set $\{n: x(n)\:R\:y(n)\}$.

\begin{lemma} \label{meng_forcing}
Suppose that $\F=\F^+$ is a  semifilter with Menger square. Let
$x$ be   $\IP_\F^*$-generic, $\IQ\in V[x]$ be an
$\w^\w$-bounding poset, and $H$ be a $\IQ$-generic over $V[x]$.
Then
in $V[x*H]$ there is no  semifilter $\G=\G^+$  containing $\F$
such
that $\G^2$ is Menger.
\end{lemma}
\begin{proof}
Throughout the proof we shall identify $x$ with $\cup x:\w\times\w\to 2$.
Suppose to the contrary that such a $\G$ exists. Set
$x_j(n)=x(j,n)$. In $V[x*H]$, the following 2 cases are possible.

a). For every $m\in\w$ there exists $k>m$ such that
$\bigcup_{j\in[m,k)}[x_j=x_m] \in\G$. Then we can inductively
construct an increasing sequence $\la m_k:k\in\w\ra$ such that
\begin{equation} \label{eq_1}
\bigcup_{j\in[m_k, m_{k+1})}[x_j=x_{m_k}]\in\G \mbox{ \ for all \
}
k.
\end{equation}
 Since $\IP^*_\F*\IQ$ is $\w^\w$-bounding, we may additionally
assume that this sequence
is in $V$.

b). There exists $m$ such that $\bigcup_{j\in
[m,k)}[x_j=x_m]\in\sim\G$ for all $k>m$. This means that
$\bigcap_{j\in [m,k)}[x_j\neq x_m] \in \G$ for all $k>m$. Then
$$
[x_i= x_{i+1}]\supset [x_i\neq x_{m}] \cap [x_{i+1}\neq x_{m}]
\supset \bigcap_{j\in [m,i+2)} [x_j \neq x_m] \in \G
$$
for all $i>m$. Thus the sequence $m_k=m+1+2k$ satisfies
(\ref{eq_1}), and hence there always exists a sequence $\la
m_k:k\in\w\ra\in V$ satisfying (\ref{eq_1}).

Set $A_k=\bigcup_{j\in[m_k, m_{k+1})}[x_j=x_{m_k}]\in\G$ and
$\U_k=\{U^k_n: n\in\w\}$, where
$$U^k_n=\{\la X,Y \ra\in\mathcal P(\w)^2 : \forall i\leq k\: 
\big((X\cap A_i\cap [k,n)\neq\emptyset)\wedge
(Y\cap A_i\cap [k,n)\neq\emptyset)\big) \}.   $$ Since
$A_k\in\G=\G^+$ for all $k$, $\U_k$ is easily seen to be an open
cover of $\G^2$.  The Menger property of $\G^2$ yields a strictly
increasing $f \in\w^\w\cap V[x*H]$ such that
$\{U^k_{f(k)}:k\in\w\}$
covers $\G^2$. Since $\IP^*_\F*\IQ$ is $\w^\w$-bounding, we could
additionally assume that $f\in V$. Set $h(0)=f(0)+1$ and
$h(l+1)=f(h(l))+1$ for all $l$. Note that $U^k_n=(W^k_n)^2$,
where
$$ W^k_n=\{X \in\mathcal P(\w) : \forall i\leq k\:  (X\cap
A_i\cap [k,n)\neq\emptyset)\}.$$
Therefore there exists $\epsilon\in 2$ such that
$$O_\epsilon:=\bigcup\{ W^k_{f(k)}:
k\in\bigcup_{l\in\w}[h(2l+\epsilon),h(2l+\epsilon+1))
\}\supset\G:$$
If there were $X_\epsilon\in\G\setminus O_\epsilon$ for all
$\epsilon\in 2$, then $\la X_0, X_1 \ra$ could not be an  element
of
$U^k_{f(k)}$ for any $k$ thus contradicting the choice of $f$.
Without loss of generality $\epsilon=0$ is as above.

\begin{claim} \label{good_V}
Let $\delta\in\w^\w$ be strictly increasing. Then
$$ A_\delta:= \bigcup_{i\in\w} A_i\cap [h(2\delta(i)),
h(2\delta(i+1)))  \in \G. $$
\end{claim}
\begin{proof}
Given any $X\in\G$,  find $l\in\w$ and $k\in [h(2l),h(2l+1))$
such
that $X\in  W^k_{f(k)}$. Let $i\in\w$ be such that $l\in
[\delta(i),
\delta(i+1))$. Note that $i\leq l\leq k$,  hence $X\in 
W^k_{f(k)}$
implies $ X\cap A_i\cap [k,f(k))\neq\emptyset $. It follows that
$$ [k,f(k))\subset [h(2l), f(h(2l+1)))\subset[h(2l),
h(2l+2))\subset [h(2\delta(i)),h(2\delta(i+1))), $$
consequently $ X\cap A_i\cap
[h(2\delta(i)),h(2\delta(i+1)))\neq\emptyset $, which implies
$\bigcup_{i\in\w} A_i\cap [h(2\delta(i)), h(2\delta(i+1)))  \in
\G^+$.
\end{proof}

Let us fix any $p\in\IP^*_\F$ and set $B_i=\w\setminus
\supp(p_{m_{i+1}})\in\F^+$. By Lemma~\ref{bo_menger} used in $V$
there exists an increasing $\delta$ such that $
B:=\bigcup_{i\in\w}
B_i\cap [h(2\delta(i)), h(2\delta(i+1)))  \in \F^+=\F. $ For
every
$m\in\w$ find $i$ such that $m\in [m_i,m_{i+1})$ and set
$$q_m=p_m\cup \big(B_i\cap [h(2\delta(i)),
h(2\delta(i+1)))\times\{0\}\big) $$
if $m=m_i$ and
$$q_m=p_m\cup \big( B_i\cap [h(2\delta(i)),
h(2\delta(i+1)))\times\{1\}\big) $$
otherwise. This $q$ obviously forces (i.e., any condition in
$\IP^*_\F*\IQ$ whose first coordinate is $q$ forces) that
$B_i\cap\name{A_i} \cap [h(2\delta(i)),
h(2\delta(i+1)))=\emptyset $
for all $i$, and hence it also forces $B\cap
\name{A_\delta}=\emptyset$. Thus the set of those $q\in\IP^*_\F$
which force $B\cap \name{A_\delta}=\emptyset$ is dense, which
means
that $B\cap A_\delta=\emptyset$ (here
$A_\delta=\name{A_\delta}^{G*H}$). However,  $B\in\F\subset\G$ by
the choice of $\delta$ and $A_\delta\in\G$ by Claim~\ref{good_V},
and therefore $B\cap A_\delta=\emptyset$ contradicts $\G=\G^+$.
This
contradiction completes our proof.
\end{proof}

\noindent\textit{Proof of Theorem~\ref{main_men}.} \ Suppose that
there exists a topological group $G$ and a  compactification
$bG$,
such that $(bG\setminus G)^2$ is Menger but not $\sigma$-compact.
Then
 in the same way
as at the beginning of the proof of Theorem~\ref{main_hur} we
conclude that $G$ is feathered. By Lemma~\ref{feathered} we may
assume without loss of generality   that $G$ satisfies the
premises
of Lemma~\ref{towards_semifilter}. Applying this lemma for
$\mathsf
P$ being the property of having the Menger square we get a
semifilter $\F=\F^+$ such that $\F^2$ is Menger. Thus the theorem
will be proved as soon as we construct a model of ZFC in which
there
are no semifilters $\F=\F^+$ with Menger square.

To this end let us assume that GCH holds in $V$ and consider a
function $B:\w_2\to H(\w_2)$, the family of all sets whose
transitive closure has size $<\w_2$, such that for each $x\in
H(\w_2)$ the family $\{\alpha:B(\alpha)=x\}$ is
$\w_1$-stationary.
Let $\la \IP_\alpha,\name{\IQ}_\beta:\beta<\alpha\leq\w_2\ra$ be
the
following  iteration with at most countable supports: If
$B(\alpha)$
is a $\IP_\alpha$-name for $\IP^*_{\name{\F}}$ for some
semifilter
$\name{\F}$ such that $\forces_{\IP_\alpha}$
``$\name{\F}=\name{\F}^+$ and $\name{\F}^2$ is Menger'', then
$\name{\IQ}_\alpha=\IP_{\name{\F}}^*$. Otherwise we let
$\name{\IQ}_\alpha$ to be a $\IP_\alpha$-name for the trivial
forcing. Then $\IP_{\w_2}$ is $\w^\w$-bounding forcing notion
with
$\w_2$-c.c. being a countable support iteration of length $\w_2$
of
proper $\w^\w$-bounding posets of size $\w_1$  over a model of
CH.

Let $G$ be a $\IP_{\w_2}$-generic over $V$ and suppose that
$\F\in
V[G]$ is a semifilter such that $\F=\F^+$ and $\F^2$ is Menger.
Then
the set $\big\{\alpha:\F_\alpha:=(\F\cap V[G\cap\IP_\alpha])\in
V[G\cap\IP_\alpha], $ $\F_\alpha=\F^+_\alpha$  and $\F_\alpha^2$
is
Menger in $V[G\cap\IP_\alpha]  \big\}$ contains an $\w_1$-club
subset of $\w_2$, and hence for one of these $\alpha$ we have
that
$\name{\IQ}_\alpha=\IP^*_{\name{\F}_\alpha}$, where
$\name{\F}_\alpha$ is a $\IP_\alpha$-name such that
$\name{\F}_\alpha^{G\cap\IP_\alpha}=\F\cap V[G\cap\IP_\alpha]$.
Now,
a direct application of Lemmata~\ref{meng_forcing} and
\ref{w^w-bound} implies that $\F_\alpha\subset\F$ cannot be
enlarged
to any semifilter $\U\in V[G]$ such that $\U^2$ is Menger and
$\U^+=\U$, which contradicts our choice of $\F$. \hfill $\Box$
\medskip

Let us note that in the proof of Theorem~\ref{main_men} above
we have also proven the following 

\begin{theorem}
It is consistent with ZFC that there are no semifilters $\F$
such that $\F=\F^+$ and $\F^2$ is Menger.  
\end{theorem}

\section{On a possible dichotomy for the Menger
property}\label{dichotomy_menger}

Our first attempt to find a counterpart of the Menger property is
based on its game characterization we have exploited in
Section~\ref{menger}. As the Menger game produces a strengthening
of
the Lindel\" of property, we should consider a game which
produces a
strengthening of the Baire property.

There is an obvious candidate for this purpose: the Banach-Mazur
game, see for instance \cite{Kec95} for more information. This
game
${\text {BM}}(X)$ is played on the space $X$ in $\omega$-many
innings between two players $\alpha $ and $\beta$ as follows.
$\beta$ makes the first  move by
 choosing  a non-empty open set $U_0$ and $\alpha $ responds by
taking a non-empty open set  $V_0\subseteq U_0$.   In general, at
the n-th  inning $\beta  $ chooses a non-empty open set
$U_n\subseteq V_{n-1}$ and $\alpha  $ responds by taking a
non-empty
open set $V_n\subseteq U_n$. The rule is that $\alpha $ wins if
and
only if  $\bigcap \{V_n:n<\omega\}\ne \emptyset $. The
relationship
of the Banach-Mazur game with  Baire spaces is given by the
following \cite[Theorem 8.11]{Kec95}.

\begin{theorem}
A space $X$ is Baire  if and only if player $\beta$ does not have
a
winning strategy in ${\text{BM}}(X)$.
\end{theorem}

Consequently,  if $\alpha $ has a winning strategy, then the
space
is  Baire.
\medskip

\noindent\textbf{Definition.}  A  space $X$ is \emph{weakly
$\alpha
$-favorable} if  player $\alpha $ has a winning strategy in the
Banach-Mazur game. $X$ is said to be \emph{$\alpha $-favorable} 
if
player $\alpha $ has a winning tactic, i. e. a winning strategy
depending only on the last move of $\beta$. \hfill $\Box$
\medskip

Every pseudocompact space is $\alpha $-favorable: player $\alpha
$
has an easy winning tactic by choosing for any    $U_n$ a
non-empty
open set $V_n$ such that $\overline {V_n}\subseteq U_n$.
 Of course,  every weakly $\alpha
$-favorable space is Baire. Moreover, the following observation
shows that being weakly $\alpha$-favorable often contradicts the
Menger property.

\begin{observation}\label{theorem12}
 No  nowhere locally compact  weakly $\alpha$-favorable
subset $X$ of the real line is Menger.
\end{observation}
\begin{proof}
Since $X$ is nowhere locally compact, we may assume that
$X\subset\mathbb R\setminus\mathbb Q$, and the  latter we shall
identify with $\w^\w$. By \cite[Theorem~8.17(1)]{Kec95} $X\supset
Y$
for some dense $G_\delta$ subset $Y$ of $\w^\w$. By the Baire
category theorem $Y$ cannot be contained in a $\sigma$-compact
subspace of $\w^\w$, and hence it contains a copy $Z$ of $\w^\w$
which is closed in $\w^\w$ according to \cite[Corollary
21.23]{Kec95}. Therefore $Z$ is a closed in $X$ copy of $\w^\w$,
which implies that $X$ is not Menger as the Menger property is
inherited by closed subspaces.
\end{proof}

 Therefore, weak  $\alpha $-favorability   seems to be 
 a good candidate to be the counterpart of the Menger property.
However, this is not the case by Theorem~\ref{theorem11} below.
Let
us recall that  a set $S\subset \mathbb R$ is a Bernstein set
provided that  both $S$ and $\mathbb R\setminus S$ meet every
closed
uncountable subset of $\mathbb R$. The next two lemmas seem to be known,
but we were not able to find them in the literature. That is why we 
present their proofs. 

\begin{lemma}\label{lemma8}    There is a subgroup $G$   of the
real
line $\mathbb  R$ containing the rationals which is a Bernstein set.
\end{lemma}
\begin{proof}
Let $\{C_\alpha :\alpha <\mathfrak c\}$ be the collection of all
closed uncountable subsets of $\mathbb R$.
 Here, we will consider $\mathbb R$
as a $\mathbb Q$-vector space. Choose a point $x_0\in C_0$ and
denote by $G_0$ the vector subspace   of $\mathbb R$ generated 
by
$\{1, x_0\}$. Obviously, we have $|G_0|=\w$. Then,   pick a point
$
y_0\in C_0\setminus G_0$. We proceed by transfinite induction, by
assuming to have already constructed a non decreasing family of
vector subspaces $\{G_\beta:\beta<\alpha  \}$ of $\mathbb R$
satisfying $|G_\beta|\le |\beta|+ \w$ for each $\beta$ and points
$x_\beta, y_\beta\in C_\beta$ in such a way that $x_\beta\in
G_\beta$ and $\{ y_\beta: \beta<\alpha \}\cap
\bigcup\{G_\beta:\beta<\alpha \}=\emptyset$. The set $H_\alpha
=\bigcup\{G_\beta:\beta<\alpha \}$ has cardinality not exceeding
$|\alpha |+\w$ and  therefore even the vector subspace $K_\alpha$
generated by the set $H_\alpha \cup \{y_\beta:\beta<\alpha \}$
has
cardinality  less than $\mathfrak c$. So we may pick a point
$x_\alpha
\in C_\alpha \setminus K_\alpha $. Then, let $G_\alpha $ be the
vector subspace generated by $H_\alpha \cup \{x_\alpha \}$ and
finally pick a point $y_\alpha \in C_\alpha \setminus G_\alpha $.
It
is clear that $|G_\alpha |\le |\alpha |+\w$.  To complete the
induction, we need to show $y_\beta\notin G_\alpha $ for each
$\beta<\alpha $.  Indeed, if we had $y_\beta\in G_\alpha $ for
some
$\beta$, then $y_\beta=z+qx_\alpha $, where  $z\in H_\alpha $ and
$q\in \mathbb Q\setminus \{0\}$. But, this would imply $x_\alpha
=q^{-1}z-q^{-1}z\in K_\alpha $, in contrast with the way
$x_\alpha $
was chosen.

Now, we let $G=\bigcup\{G_\alpha :\alpha <\mathfrak c\}$. It is
clear
that $G$  is a $\mathbb Q$-vector subspace,  and hence a 
subgroup, 
of $\mathbb R$ which is also a Bernstein set.
\end{proof}

\begin{lemma} \label{lemma9}
A Bernstein    set $X\subset \omega^\omega$ does not have the
Menger
property.
\end{lemma}
\begin{proof}
 For any
$f\in \omega^\omega$ there exists some $g\in X$ such that
$f(n)<g(n)$ for each $n\in \omega$. This comes from the fact that
$X$ must meet the Cantor set  $\prod _{n<\omega} \{
f(n)+1,f(n)+2\}$. To finish, recall that a dominating subset of
$\omega^\omega$ is never Menger. Indeed, for any $n<\omega$ let
$\pi_n:\omega^\omega\to \omega$ be the projection onto the n-th
factor and put
 $\mathcal U_n=\{\pi_n^{-1}(k)\cap X :k\in \omega\}$. Each
$\mathcal U_n$
is an open cover of $X$. For any choice  of a finite set
$\mathcal
V_n\subseteq \mathcal U_n$, we may define a function $g: \omega
\to
\omega $ by letting $g(n)=\max\pi_n(\bigcup  \mathcal V_n)$, if
$\mathcal V_n\ne\emptyset$, and $g(n)=0$ otherwise. Since $X$ is
dominating, there is some $f\in X$ such that $g(n)<f(n)$ for each
$n$. Clearly, $ f\notin \bigcup\{\bigcup \mathcal V_n:n<\omega\}$
and so $X$ is not Menger.
\end{proof}

\begin{lemma}\label{lemma10} A Bernstein set $X\subseteq \mathbb
R$  is
not weakly $\alpha $-favorable.
\end{lemma}
\begin{proof}
By \cite[Theorem~8.17(1)]{Kec95} any weakly $\alpha$-favorable
subspace of $\mathbb R$ is comeager, while no Bernstein set can
be
comeager because any comeager subspace of $\mathbb R$ contains
homeomorphic copies of the Cantor set.
\end{proof}

These three lemmas   imply:

\begin{theorem}\label{theorem11}
There exists a topological group $G$ and its compactification $bG$
such that the remainder $bG\setminus G$ is neither Menger nor
weakly $\alpha $-favorable.
\end{theorem}
\begin{proof} Recall that the set of irrationals $\mathbb
R\setminus \mathbb Q$
is homeomorphic to $\omega^\omega$. Let $G$ be such as  in
Lemma~\ref{lemma8}. By Lemma~\ref{lemma9}  $\mathbb R\setminus
G\subseteq \omega^\omega$ is not Menger, and by
Lemma~\ref{lemma10}
 $\mathbb R\setminus G$ is not weakly $\alpha $-favorable. Now,
it
suffices to take as $bG$   the compactification of $\mathbb R$
obtained by adding two end-points.
\end{proof}

Theorem~\ref{theorem11} implies that  the counterpart of the
Menger
property should be in between of weakly  $\alpha $-favorable and
Baire.

\section{Miscellanea}

A very important example of a topological group is $C_p(X)$, the
subspace of $\mathbb R^X$ with the Tychonoff product topology
consisting of all continuous functions. We expect that the
remainder
of $C_p(X)$ cannot distinguish between being Menger and
$\sigma$-compact, but we cannot prove this.

\begin{question}\label{q_cp}
Is it true that a remainder of $C_p(X)$ is Menger if and only if
it
is $\sigma$-compact?
\end{question}

Below we present some results giving a partial solution of
Question~\ref{q_cp}.

\begin{proposition} \label{prop_end}
Let $Z$ be a compactification of $C_p(X)$. If $Z\setminus
C_p(X)$ is Menger, then $C_p(X)$ is first countable and
hereditarily
Baire.
\end{proposition}
\begin{proof}
 Since  Menger spaces are Lindel\"of, by  Henriksen-Isbell's
theorem  \cite{HenIsb57}, $C_p(X)$ is of countable type,
 and therefore it contains a compact subgroup with countable
outer
 base according to \cite[Lemma~4.3.10]{ArkTka11}. It is easy to
see
 that there is no compact subgroup of $C_p(X)$ except for
 $\{0\}$: for any $f\in C_p(X)\setminus\{0\}$, the set
 $\{nf:n\in\w\}$ is not contained in any compact
 $K\subset C_p(X)$ because $\{nf(x):n\in\w\}$ is unbounded
 in $\mathbb R$ if $f(x)\neq 0$. Therefore $C_p(X)$ is
 first-countable, and hence $X$ is countable.

Since $Z\setminus C_p(X)$ is Menger, it follows that $C_p(X)$
contains no closed copy of $\mathbb Q$. Now, a theorem of Debs
\cite{Deb88} implies that $C_p(X)$ is hereditarily Baire.
\end{proof}

The following fact together with Proposition~\ref{prop_end}
 gives
the positive answer to Question~\ref{q_cp} for spaces containing
non-trivial convergent sequences.

\begin{observation}
If $X$ contains a non-trivial convergent sequence then $C_p(X)$
is
not Baire.
\end{observation}
\begin{proof}
There is  nice characterization of the Bairness for spaces of
the
form $C_p(X)$ due to Tkachuk. However, we shall present here a
direct elementary proof. Suppose that $(x_n)_{n\in\w}$ is an
injective sequence converging to $x$. Set $$F_n=\{f\in
C_p(X):\forall m\geq n\: (|f(x)-f(x_m)|\leq 1)\}.$$ It is easy to
check that each $F_n$ is closed nowhere dense in $C_p(X)$ and
$C_p(X)=\bigcup_{n\in\w}F_n$.
\end{proof}

 By a theorem of Lutzer (see Problem 265 in \cite{Tka14}),
$C_p(X)$ is \v{C}ech-complete if and only if $X$ is countable and
discrete. So to answer    Question~\ref{q_cp} in the affirmative
we
need  to show that $C_p(X)$ has a Menger remainder  only if $X$
is
countable and discrete.

Note that in the proof of the ``if'' part of
Theorem~\ref{main_sch}
the compactification was a topological group itself, namely
$(\mathcal P(\w),\Delta)$. We do not know whether complements to
Menger subspaces in other Polish  groups (e.g., $\mathbb R$) may
consistently be subgroups. The next proposition imposes some
restrictions.

\begin{proposition}
Let $G$ be an analytic  topological group and $M$ be a non-empty
Menger subspace of $G$. If $G\setminus M$ is a subgroup of $G$,
then
$G$ is $\sigma$-compact and $M$ contains a topological copy of
$\mathcal P(\w)$.
\end{proposition}
\begin{proof}
Suppose that $H=G\setminus M$ is a subgroup of $G$ and fix $g\in
M$.
Then $H\subset g^{-1}*M $, where $*$ is the underlying operation
on
$G$. Therefore $G=M\cup g^{-1}*M$ is Menger, and hence it is
$\sigma$-compact, see \cite{Arh86}.

Now suppose that $M$ contains no topological copy of $\mathcal
P(\w)$ and let $X\subset G$ be homeomorphic to $\mathcal P(\w)$.
If
$X\subset H$ then $g*X\subset g^*H\subset M$ which is impossible
by
our assumption above. Thus $X\cap M\neq\emptyset$. Since $M$
contains no copy of $\mathcal P(\w)$, $X\setminus M$ is dense in
$X$,  and hence there exists a countable dense subset $Q$ of $X$
disjoint from $M$. Then $X\cap M=(X\setminus Q)\cap M$ is a
closed
subset of $M$. Note that $(X\setminus Q)$ is a copy of $\w^\w$
and
$M\cap (X\setminus Q)$ is a Bernstein set in $(X\setminus Q)$. To
finish, it suffices to apply Lemma~\ref{lemma9}.
\end{proof}

The following statement shows that the classical Cantor-Bendixon
inductive procedure does not have any variant allowing to
separate a
``nowhere perfect'' core of a Menger space from its
``$\sigma$-compact part''.

\begin{proposition} \label{proposition19}
 There exists a Baire dense nowhere locally
compact subgroup $\I $ of $\mathcal P(\w)$ with the Menger
property
such that for every $\sigma$-compact subspace $\mathcal  S$ of
$\I $
there exists $\K \subset \I $ homeomorphic to $\mathcal P(\w)$
such
that $\K \cap \mathcal S=\emptyset$.
\end{proposition}
\begin{proof}
It is well-known that there exists a non-meager Menger filter
$\F$
on $\w$, see, e.g., the proof of Theorem~1 in \cite{RepZdoZha14}.
Let $\mathcal I $ be the dual ideal of $\F$. Then $\mathcal I $
is
Menger, nowhere locally compact, and non-meager being
homeomorphic
to $\F$. Also, $\mathcal I $ is a subgroup of $\mathcal P(\w)$,
and
hence it is Baire because each non-meager topological group is
so.
Note that $\mathcal I $ contains copies of $\mathcal P(\w)$:  for
every infinite $I\in\mathcal I $ the set $\mathcal
P(I)\subset\mathcal I $ is such a copy. Let us fix $\mathcal
X\subset\mathcal I $ homeomorphic to $\mathcal P(\w)$ and a
$\sigma$-compact $\mathcal S\subset\mathcal I $. Then there
exists
$I\in \mathcal I \setminus (\mathcal S+\mathcal  X)$ because
$\mathcal S+\mathcal X$ is $\sigma$-compact and $\mathcal I $ is
not. It follows that $\mathcal K:=\{I\}-\mathcal  X$ is a copy of
$\mathcal P(\w)$ disjoint from $\mathcal S$.
\end{proof}

\noindent\textbf{Acknowledgement.}  The authors wish to thank
Masami
Sakai and Boaz Tsaban for   many useful comments.
We also thank the anonymous referee for a careful reading and many
suggestions which improved the paper.


\begin{thebibliography}{ChGP??}
\bibitem{Arh86}
Arhangel'skii, A., {\it Hurewicz spaces, analytic sets and fan
tightness of function spaces,} Doklady Akademii Nauk SSSR \textbf{287}
(1986),  525--528. (In Russian).

\bibitem{ArhMil14}  Arhangel'skii, A.; van Mill, J., {\it
Topological homogeneity,}
 In:
Recent Progress in General Topology III (KP Hart, J. van Mill, P.
Simon (eds.)), Springer 2014, pp. 1-68.

\bibitem{ArkTka11} Arhangel'skii, A.; Tkachenko, M.,
{\it Topological groups and related structures.} Atlantis Studies
in
Mathematics, 1. Atlantis Press, Paris; World Scientific
Publishing
Co. Pte. Ltd., Hackensack, NJ, 2008.

\bibitem{Arh08} Arhangel'skii, A., {\it Two types of remainders
of topological groups,}
  Commentationes Mathematicae Universitatis
Carolinae  \textbf{49}  (2008),    119--126.


\bibitem{Arh09} Arhangel'skii, A., {\it The Baire property in
remainders of topological groups
and other results,}   Commentationes Mathematicae Universitatis
Carolinae  \textbf{50}  (2009),    273--279.

\bibitem{Aur10} Aurichi, L.F.,
{\it $D$-spaces, topological games, and selection principles,}
Topology Proceedings \textbf{36} (2010), 107--122.

\bibitem{AurBel??} Aurichi, L.F.; Bella, A.,
{\it When is a space Menger at infinity?}, 
 Applied General Topology   \textbf{16}  (2015),   75--80.

\bibitem{BanZdo06} Banakh, T.; Zdomskyy, L.,  {\it Selection
principles and infinite games on
multicovered spaces,}  in: Selection Principles and Covering
Properties in Topology (Lj. D.R. Kocinac, ed.), Quaderni di
Matematica 18, Dept. Math., Seconda Universita di Napoli, Caserta
(2006), 1--51.

\bibitem{BlaHruVer13}
Blass, A.; Hru\v{s}\'ak, M.; Verner, J., {\it On strong P-points}, 
 Proceedings of the American Mathematical Society  \textbf{141}  (2013),   2875--2883.


\bibitem{Can88} Canjar, R.M., {\it Mathias forcing which does not
add dominating reals,}
Proceedings of the American Mathematical Society \textbf{104}
(1988),  1239--1248.

\bibitem{ChoRepZdo??} Chodounsky, D.; Repov\v{s}, D.; Zdomskyy,
L.,
{\it  Mathias forcing and combinatorial covering properties of
filters,} Journal of Symbolic Logic \textbf{80} (2015), 1398--1410. 

\bibitem{Deb88} Debs, G. Espaces h�r�ditairement de Baire.
(French) [Hereditarily Baire spaces] \
Fundamenta Mathematicae \textbf{129} (1988),  199--206.


\bibitem{Eng} Engelking, R., {\it General topology.}  Monografie
Matematyczne,  Vol. 60.
 PWN---Polish Scientific Publishers, Warsaw, 1977.

\bibitem{HenIsb57} Henriksen, M.; Isbell, J. R., {\it Some
properties of compactifications,}
 Duke Mathematical Journal \textbf{25} (1957), 83--105.

\bibitem{HruMin??} Hru\v{s}\'ak, M.; Minami, H.,
{\it Laver-Prikry and Mathias-Prikry type forcings,}
 Annals of Pure and Applied  Logic \textbf{165} (2014), 880--894.

\bibitem{Hur25} Hurewicz, W.,
{\it \"Uber die Verallgemeinerung des Borellschen Theorems},
Math.
Z. \textbf{24} (1925), 401--421.

\bibitem{COC2} Just, W.; Miller, A.W.; Scheepers, M.; Szeptycki,
P.J., {\it The combinatorics of open covers. II,} Topology and
its
Applications \textbf{73} (1996),  241--266.


\bibitem{Kec95} Kechris, A.S.,
{\it Classical descriptive set theory.} Graduate Texts in
Mathematics, 156. Springer-Verlag, New York, 1995.

\bibitem{RepSem98} Repov\v{s}, D.; Semenov, P.,
 {\it Continuous selections of multivalued mappings.}
Mathematics and its Applications, 455. Kluwer Academic
Publishers,
Dordrecht, 1998.

\bibitem{RepZdoZha14} Repov\v{s}, D.; Zdomskyy, L.; Zhang, S.,
{\it Countable dense homogeneous filters and the Menger covering
property,} Fundamenta Mathematicae \textbf{224} (2014), 233--240.

\bibitem{SamSchTsa09}
 Samet, N.; Scheepers, M.; Tsaban, B.,
{\it Partition relations for Hurewicz-type selection hypotheses},
 Topology and its Applications \textbf{156} (2009), 616--623.


\bibitem{COC1} Scheepers, M.,
{\it Combinatorics of open covers. I.  Ramsey theory,}   Topology
and its Applications \textbf{69}  (1996),   31--62.

\bibitem{She_propimp} Shelah, S., {\it Proper and improper
forcing.}
 Second edition. Perspectives in Mathematical Logic.
Springer-Verlag, Berlin, 1998.

\bibitem{Tka14} Tkachuk, V., {\it
    A Cp-Theory Problem Book.
    Special Features of Function Spaces.}
    Springer International Publishing, 2014.

\bibitem{Tsa11} Tsaban, B., {\it Menger's and Hurewicz's
problems: solutions from "the book'' 
and refinements}. 
 Set theory and its applications,  
211--226, Contemp. Math., 533, Amer. Math. Soc., Providence, RI,
2011.


\bibitem{TsaZdo08} Tsaban, B.; Zdomskyy, L.,
{\it Combinatorial images of sets of reals and semifilter
trichotomy,}  Journal of Symbolic Logic \textbf{73}
 (2008), 1278--1288.

 \bibitem{Zdo06} Zdomskyy, L., {\it $o$-boundedness of free
objects over a
 Tychonoff space,} Matematychni Studii \textbf{25}  (2006),
10--28.

\bibitem{Zdo05} Zdomskyy, L.,  {\it A semifilter approach to
selection principles,}
  Commentationes Mathematicae Universitatis Carolinae \textbf{46}
(2005), 525--539.

\end{thebibliography}
\end{document}